\documentclass[letterpaper, 10 pt, conference]{ieeeconf}  

\IEEEoverridecommandlockouts                              

\usepackage{graphicx}
\usepackage[dvipsnames]{xcolor}
\graphicspath{{./figures/}}

\usepackage{amsmath,amssymb,amsfonts,mathtools}

\usepackage{mathrsfs}
\usepackage{amsfonts,bbm}
\usepackage{dsfont}
\usepackage{epstopdf}
\usepackage{subfigure}
\usepackage[acronym]{glossaries}
\usepackage{balance}
\usepackage{url}
\usepackage{placeins}
\usepackage{textgreek} 
\usepackage{tikz}
\usepackage{pgfplots}
\pgfplotsset{compat=1.10}
\usepgfplotslibrary{fillbetween}

\newtheorem{theorem}{Theorem}
\newtheorem{definition}{Definition}

\newtheorem{lemma}{Lemma}

\newtheorem{remark}{Remark}
\newtheorem{assumption}{Assumption}
\newtheorem{standing}{Standing Assumption}

\newcommand{\R}{\mathbb{R}}

\newcommand{\N}{\mathbb{N}}

\newcommand{\mc}[1]{\mathcal{#1}}

\newcommand{\argmin}{\mathrm{argmin}}

\newcommand{\proj}{\mathrm{proj}}
\newcommand{\diag}{\mathrm{diag}}

\newcommand{\col}{\textrm{col}}

\newcommand{\bs}{\boldsymbol}

\newcommand\oprocendsymbol{\hbox{$\square$}}
\newcommand\oprocend{\relax\ifmmode\else\unskip\hfill\fi\oprocendsymbol}

\newcommand{\nc}{\mathrm{N}}
\newcommand{\0}{\mathbf{0}}
\newcommand{\1}{\mathbf{1}}

\newcommand{\Rmnum}[1]{\expandafter\@slowromancap\romannumeral #1@}
\makeatletter

\makeglossaries
\newacronym{GNEP}{GNEP}{generalized Nash equilibrium problem}
\newacronym{NE}{NE}{Nash equilibrium}
\newacronym{NEP}{NEP}{Nash equilibrium problem}
\newacronym{GNE}{GNE}{generalized Nash equilibrium}
\newacronym{v-GNE}{v-GNE}{variational GNE}
\newacronym{ISS}{ISS}{Input-to-state-stable}
\newacronym{KKT}{KKT}{Karush-Kuhn-Tucker}

\usepackage{algorithm}
\usepackage{algorithmicx}
\makeatletter
\renewcommand{\ALG@name}{Algorithm}
\makeatother
\usepackage{empheq}
\let\labelindent\relax
\usepackage{enumitem}

\def\QEDhereeqn{\eqno\let\eqno\relax\let\leqno\relax\let\veqno\relax\hbox{\QED}}
\def\QEDopenhereeqn{\eqno\let\eqno\relax\let\leqno\relax\let\veqno\relax\hbox{\QEDopen}}

\title{\LARGE \bf
A continuous-time distributed generalized Nash equilibrium seeking algorithm  over networks for double-integrator agents
}

\author{Mattia Bianchi and Sergio Grammatico
\thanks{The authors are with the Delft Center for Systems and Control (DCSC), TU Delft, The Netherlands
       E-mail addresses: \texttt{\{m.bianchi, s.grammatico\}@tudelft.nl}. This work was partially supported by NWO under research project OMEGA (grant n. 613.001.702) and by the ERC under research project COSMOS (802348).} 
}

\begin{document}

\maketitle
\thispagestyle{empty}
\pagestyle{empty}

\begin{abstract}
	We consider
	a system of single- or double-integrator agents playing a generalized Nash game over a network, in a partial-information scenario.
	We address the generalized Nash equilibrium seeking problem by  designing a fully-distributed dynamic controller, based on continuous-time consensus and primal-dual gradient dynamics.
	Our main technical contribution is to show convergence of the closed-loop system to a variational equilibrium, under strong monotonicity and Lipschitz continuity of the game mapping, by leveraging monotonicity properties and stability theory for projected dynamical systems.
\end{abstract}

\section{Introduction}\label{sec:introduction}
 \Gls{GNE} problems arise in several network systems,  where multiple selfish decision-makers, or agents, aim at optimizing their individual, yet inter-dependent, objective functions, subject to shared constraints. Engineering applications include demand-side management in the smart grid \cite{Saad2012}, charging/discharging of electric vehicles \cite{Grammatico2017},  
formation control \cite{Stankovic_Johansson_Stipanovic_2012} and
communication networks \cite{Palomar_Eldar_Facchinei_Pang_2009}.
From a game-theoretic perspective, the aim is to design distributed \gls{GNE} seeking algorithms, using the local information available to each agent.
%
%
Moreover, in the cyber-physical systems framework, games are often played by agents with their own dynamics \cite{DePersisMonshizadeh_2019},  \cite{Stankovic_Johansson_Stipanovic_2012},
and controllers have to be conceived to steer the physical process to a Nash equilibrium, while ensuring closed-loop stability. 
This stimulates the development of  continuous-time schemes \cite{GadjovPavel2018}, \cite{DePersisGrammatico_TAC2020}, for which control-theoretic properties are more easily unraveled.

\emph{Literature review:} 
A variety of different methods have been proposed to seek \gls{GNE} in a distributed way \cite{YiPavel2019}, \cite{Yu_VanderSchaar_Sayed_2017}, \cite{BelgioiosoGrammatico_ECC_2018}.
These works refers to a full-information setting, where each agent can access the decision of all other agents, for example if a coordinator broadcasts the data to the network.
Nevertheless, there are applications where the existence of a central node must be excluded and each agent only relies on the information exchanged over a network, via peer-to-peer communication. 
To deal with this partial-information scenario, payoff-based algorithms for \gls{NE} seeking have been studied,  \cite{FrihaufKrsticBasar_2012},  \cite{Stankovic_Johansson_Stipanovic_2012}. 
In this paper, we are instead interested in a different,
model-based, approach. We assume that the agents agree on sharing their strategies with their neighbors;  each agent keeps an estimate of other agents' actions and asymptotically reconstructs the true values, exploiting the information exchanged over the network. 
This solution has been examined extensively for games without coupling constrains, both in discrete time \cite{SalehisadaghianiWeiPavel2019},  \cite{Koshal_Nedic_Shanbag_2016}, and continuous-time  \cite{GadjovPavel2018},  \cite{DePersisGrammatico2018}.
However, fewer works deal with generalized games. Remarkably, Pavel in \cite{Pavel2018} derived a single-timescale, fixed step sizes \gls{GNE} learning algorithm, by leveraging an elegant operator splitting approach. 
The authors in \cite{DengNian2019} proposed a continuous-time design for aggregative games with equality constraints.
All the results mentioned above consider single-integrator agents only. 
Distributively driving  a network of more complex physical systems to game theoretic solutions is still a relatively unexplored problem. 
With regard to aggregative games, a proportional integral feedback algorithm was developed in \cite{DePersisMonshizadeh_2019} to seek a \gls{NE} in  networks of passive nonlinear second-order systems. In  \cite{Zhang_et_al_NL_NE_2019}, continuous-time gradient-based controllers were introduced, for some classes of nonlinear systems with uncertainties. The authors of \cite{Stankovic_Johansson_Stipanovic_2012} considered generally coupled costs games played by linear time-invariant agents, via a discrete-time extremum seeking approach. \gls{NE} problems arising in systems of multi-integrator agents, in the presence of deterministic disturbances, were addressed in \cite{RomanoPavel2019}.
In all the references cited, the assumption is made of unconstrained action sets and absence of coupling constraints.
%
%
\newline \indent
\emph{Contribution:} Motivated by the above, in this paper we investigate continuous-time \gls{GNE} seeking for networks of single- or double-integrator agents. We consider games with affine coupling constraints, played under partial-decision information. Specifically:
\begin{itemize}
	[leftmargin=*,topsep=0pt]
 \item We introduce a primal-dual projected-gradient controller for single-integrator agents, which is a continuous-time version of the one proposed in \cite{Pavel2018}.
We show convergence of both primal and dual variables, under strong monotonicity and Lipschitz continuity of the game mapping, 
We are not aware of other continuous-time \gls{GNE} seeking algorithms for games with generally coupled costs, whose convergence is guaranteed under such mild assumptions.
With respect to the setup (for aggregative game only) in \cite{DengNian2019}, we can also handle inequality constraints.
\item We show how our controller can be adapted to learning \gls{GNE} in games with shared constraints, played by double-integrator agents. To the best of our knowledge,
this is the first equilibrium-seeking algorithm for \emph{generalized} games where the agents have second-order dynamics.
\end{itemize}
%

\smallskip
\emph{Basic notation:}
$\R$ ($\R_{\geq 0}$) denotes the set of (nonnegative) real numbers.
$\0$ ($\1$) denotes a matrix/vector with all elements equal to $0$ ($1$); we may add the dimension as subscript for clarity. $\|\cdot\|$ denotes the Euclidean  norm. $I_n\in\R^{n\times n}$ denotes the identity matrix of dimension $n$.  For a matrix $A \in \R^{n \times m}$, its transpose is $A^\top$, $[A]_{i,j}$ represents the element on the row $i$ and column $j$. $A\otimes B$ denotes the Kronecker product of the matrices $A$ and $B$.
$A \succ 0$  stands for symmetric positive definite matrix.
If $A$ is symmetric, $\lambda_{\textnormal{min}}(A):= \lambda_1(A)\leq\dots\leq\lambda_n(A)=:\lambda_{\textnormal{max}}(A)$ denote its eigenvalues.
Given $N$ vectors $x_1, \ldots, x_N$,  ${x} := \col\left(x_1,\ldots,x_N\right) = [ x_1^\top \ldots  x_N^\top ]^\top$, and for each $i=1,\dots,N$, ${x}_{-i} := \col\left(x_1,\ldots,x_{i-1},x_{i+1},\dots,x_N\right)$. 
For a differentiable function $g:\R^n \rightarrow \R$, $\nabla_{\!x} g(x)$ denotes its gradient.  
\newline
\indent
\emph{Operator-theoretic definitions:}  ${\overline{H}}$ denotes  the closure of a set $H\subseteq \R^n$.
A set-valued mapping $\mathcal{F} : \R^n \rightrightarrows \R^n$ is
($\mu$-strongly) monotone if $(u-v)^\top  (x-y) \geq 0 \, (\geq \mu \|x-y\|^2 )$ for all $x \neq y \in \R^n$, $u \in \mathcal{F} (x)$, $v \in \mathcal{F} (y)$. Given a closed, convex set $S\subseteq\mathbb{R}^n$, the mapping $\proj_{S}:\R^n \rightarrow S$ denotes the projection onto $S$, i.e., $\proj_{S}(v) := \argmin_{y \in S} \left\| y - v\right\|$.
The set-valued mapping $\nc_{S}: \R^n \rightrightarrows \R^n$ denotes the normal cone operator for the the set $S $, i.e., 
$\nc_{S}(x) = \varnothing$ if $x \notin S$, $\left\{ v \in \R^n \mid \sup_{z \in S} \, v^\top (z-x) \leq 0  \right\}$ otherwise.
The tangent cone of $S$ at a point $x\in S$ is defined as $\mathrm{T}_S(x)=\overline{{\bigcup_{\delta>0}\frac{1}{\delta}(S-x)}}$. 
$\Pi_{S}(x,v):=\proj_{\mathrm{T}_S(x)}(v)$ denotes the projection on the tangent cone of $S$ at $x$ of a vector $v\in\R^n$.
By Moreau's Decomposition Theorem \cite[Th.~6.30]{Bauschke2017}, it holds that $v=\proj_{\mathrm{T}_S(x)}(v)+\proj_{\nc_S(x)}(v)$ and $\proj_{\mathrm{T}_S(x)}(v)^\top \proj_{\nc_S(x)}(v)=0$.

\smallskip
\begin{lemma}\label{lem:minoratingproj}
	For any closed convex set $S\subseteq \R^q$, any $y,y^\prime \in S$ and any $\xi \in \R^q$, it holds that 
	$$
	\textstyle
	(y-y^\prime)^\top \Pi_{S}\left(y,\xi\right)\leq (y-y^\prime)^\top \xi.
	$$
	In particular, if $ \ \Pi_S(y,\xi)=0$, then
	$
	(y-y^\prime)^\top\xi \geq 0.
	$
	{\hfill $\square$}
\end{lemma}

\smallskip
\begin{proof}
	By Moreau's theorem, $\left(\xi-\Pi_{C}(y,\xi)\right)\in \mathrm{N}_S(y)$, hence for any $y,y^\prime\in C$,
	$
	(y^\prime-y)^\top(\xi- \Pi_{C}(y,\xi))\leq 0.
	$
\end{proof}

\smallskip
\section{Mathematical setup}\label{sec:mathbackground}

We consider a set of noncooperative agents, $ \mc I:=\{ 1,\ldots,N \}$, where each agent $i\in \mc{I}$ shall choose its decision variable (i.e., strategy) $x_i$ from its local decision set $\textstyle \Omega_i \subseteq \R^{n_i}$. Let $x = \col( (x_i)_{i \in \mc I})  \in \Omega $ denote the stacked vector of all the agents' decisions, $\textstyle \Omega = \times_{i \in \mc I} \Omega_i \subseteq \R^n$ the overall action space and $\textstyle n:=\sum_{i=1}^N n_i$. Moreover, let $ x_{-i}= \col( (x_j)_{j\in \mc I\backslash \{ i \} }) $ denote the collective strategy of the all the agents, except that of agent $i$.
The goal of each agent $i \in \mc I$ is to minimize its objective function $J_i(x_i,x_{-i})$ which depends on both the local variable $x_i$ and on the decision variables of the other agents $x_{-i}$.

{Furthermore, we consider \textit{generalized} games, 
where the agents' strategies are also coupled via some shared affine constraints. Thus
 the overall feasible set is}
\begin{align}\label{eq:feasibleset}
\textstyle
\mc{X} := \Omega \cap\left\{x \in \R^{n} \mid Ax\leq b \right\},
\end{align}
where $A:=\left[A_{1}, \ldots, A_{N}\right]$ and $b:=\sum_{i=1}^{N} b_{i}$, with $A_{i} \in \R^{m \times n_{i}}$ and $b_{i} \in \R^{m}$ being local data. 
The game then is represented by the inter-dependent optimization problems: 
\begin{align} \label{eq:game}
\forall i \in \mc{I}:
\quad \underset{y_i \in \R^{n_i}}{\argmin}  \; J_i(y_i,x_{-i}) 
\quad \text{s.t. }  \; \, (y_i,x_{-i}) \in \mc X.
\end{align} 
The technical problem we consider in this paper is the computation of a \gls{GNE}, i.e., a set of strategies from which no agent has an incentive to unilaterally deviate. 

\smallskip
\begin{definition}
	A collective strategy $x^{*}=\operatorname{col}\left((x_{i}^{*}\right)_{i \in \mathcal{I}})$ is a generalized Nash equilibrium if, for all $i \in \mc{I}$,
	\[
\quad x_{i}^{*} \in \underset{y_i}{\argmin} \, J_{i}\left(y_{i}, x_{-i}^{*}\right) \text { s.t. } (y_{i},x^*_{-i})\in \mc{X}. \QEDopenhereeqn
	\]
\end{definition}

\smallskip
Next, we postulate standard  regularity
assumptions for the constraint sets and cost functions, \cite[Ass.~1]{Pavel2018}, \cite[Ass.~1]{DePersisGrammatico2018}. 

\smallskip
\begin{standing}\label{Ass:Convexity}
	For each $i\in \mathcal{I}$, the set $\Omega_i$ is non-empty, closed and convex; $\mc{X}$ is non-empty 
	and satisfies Slater's constraint qualification;  $J_{i}$ is continuously differentiable and  $J_{i}\left(\cdot, x_{-i}\right)$ is convex for every $x_{-i}$.
	{\hfill $\square$} \end{standing}

\smallskip
Among all the possible \glspl{GNE}, we focus on the  subclass of \gls{v-GNE} \cite[Def.~3.11]{FacchineiKanzow2010}.
Under the previous assumption, $x^*$ is a \gls{v-GNE} of the game in \eqref{eq:game} if and only if there exist a dual variable $\lambda^*\in \R^m $ such that the following \gls{KKT} conditions are satisfied \cite[Th.~4.8]{FacchineiKanzow2010}:
\begin{align} \label{eq:KKT}
\begin{aligned}
\0_{n} & \in  F\left(x^{*}\right)+A^\top \lambda^{*}+\mathrm{N}_{\Omega}\left(x^{*}\right)
\\
 {\0_{m}} & \in  -\left(A x^{*}-b\right)+\mathrm{N}_{\R_{\geq 0}}^{m}\left(\lambda^{*}\right),
 \end{aligned}
\end{align}
where $F$ is the \emph{pseudo-gradient} mapping of the game:
\begin{align}
\label{eq:pseudo-gradient}
F(x):=\operatorname{col}\left( (\nabla _{\! x_i} J_i(x_i,x_{-i}))_{i\in\mathcal{I}}\right).
\end{align}
We will assume strong monotonicity of the pseudo-gradient \cite[Ass.~2]{GadjovPavel2018}, \cite[Ass.~3]{BelgioiosoGrammatico_ECC_2018}, \cite[Ass.~4]{DePersisGrammatico2018}, which is a sufficient condition for the existence of a unique \gls{v-GNE} for the game in \eqref{eq:game} \cite[Th. 2.3.3]{FacchineiPang2007}. This condition requires strong convexity of the functions $J_i(\cdot,x_{-i})$, for every $x_{-i}$, but not necessarily convexity of $J_i$ in its full argument. 

\smallskip
\sloppy
\begin{standing}\label{Ass:StrMon}
	The pseudo-gradient mapping in \eqref{eq:pseudo-gradient}  is $\mu$-strongly monotone and $\theta_0$-Lipschitz continuous, for some $\mu$, $\theta>0$: for any pair $x,y\in\R^n$, $\textstyle (x-y ) ^\top(F(x)-F(y))\geq \mu\|x-y\|^2$ and $\|F(x)-F(y)\|\leq \theta_0 \|x-y\|$.
	\hfill $\square$
 \end{standing}


\section{{Distributed generalized  equilibrium seeking}}\label{sec:distributedGNE} 
In this section, we consider the game in  \eqref{eq:game}, where each agent is associated with a dynamical system:
\begin{equation}\label{eq:integrators}
\textstyle
\forall i \in \mathcal{I}: \quad \dot{x}_i=\Pi_{\Omega_{i}}\left(x_i,u_i \right), \ x_i(0)\in \Omega_{i}.
\end{equation}
Our aim is to design the inputs $u_i$ to seek a \gls{v-GNE} in a fully distributed way. Specifically, agent $i$ does not have full knowledge of $x_{-i}$, and only relies on the information exchanged locally with neighbors over a communication network $\mathcal G(\mc{I},\mc{E})$, with weighted symmetric Laplacian $L\in \R^{N\times N}$. The unordered pair $(i,j)$ belongs to the set of edges, $\mc{E}$, if and only if agent $j$ and $i$ can exchange information.

\smallskip
\begin{standing}
	\label{Ass:Graph}
	The communication graph $\mathcal G (\mc{I},\mc{E}) $ is undirected and connected. 
	{\hfill $\square$} 
\end{standing}

\smallskip
 To cope with partial-information, each agent keeps an estimate of all other agents' actions. We denote $\bs{x}^i=\operatorname{col}((\bs{x}^i_j)_{j\in \mc{I}})\in \R^{Nn}$,  where $\bs{x} ^i_i:=x_i$ and $\bs{x}^i_j$ is $i$'s estimate of agent $j$'s action, for all $j\neq i$; 
$\bs{x}^j_{-i}=\col((\bs{x}^j_\ell)_{\ell\in\mc{I}\backslash{\{i\}}})$. Also, each agent keeps an estimate $\lambda_i\in\R^m_{\geq 0}$ of the Lagrangian multiplier and an auxiliary variable $z_i\in \R^m $ to allow for distributed consensus of the multiplier estimates. 
Our closed-loop dynamics are shown in Algorithm~\ref{algo:1}, where  $c>0$ is a global constant parameter, $W=\left[w_{i j}\right]_{i,j \in \{1,\dots,N\}} \in \R^{N\times N}$ is the weighted adjacency matrix of the graph $\mathcal G$, $\mathcal{N}_i$ is the set of neighbors of agent $i$ and $\bs{x}^i_{-i}(0)\in \R^{n-n_i},\lambda_{i}(0)\in \R^m_{\geq 0}, z_i(0)\in \R^{m}$ can be chosen arbitrarily. 

\newlength{\textfloatsepsave} \setlength{\textfloatsepsave}{\textfloatsep} \setlength{\textfloatsep}{6pt} 
\begin{algorithm}[t] \caption{Distributed GNE seeking} \label{algo:1}
		\vspace{-0.4em}
	\begin{align}
	\nonumber
	\begin{aligned}
	\dot{x}_i&=\Pi_{\Omega_{i}}\left(x_i,u_i \right)
	\\
	u_i&= - \bigl(\nabla_{\! x_{i}} J_{i}(x_{i}, \bs{x}_{-i}^{i})+A_{i}^{\top} \lambda_{i}+ 
	 c\textstyle\sum _{j \in \mathcal{N}_{i}} w_{i j}(x_{i}-\bs{x}_{i}^{j}) \bigr)
	\\
	\dot{\bs{x}}^i_{-i}&=-c\textstyle\sum_{j \in \mathcal{N}_{i}} w_{i j}  (\bs{x}_{-i}^{i}-\bs{x}_{-i}^{j} ) 
	\\
	{\dot z}_i&=\textstyle \sum_{j \in \mathcal{N}_{i}} w_{i j}(\lambda_{i}-\lambda_{j} )
	\\
	\dot\lambda_i&=
	\Pi_{\R_{\geq 0}^m}\bigl(\lambda_i, A_{i}x_{i }-b_{i}-\underset{j \in \mathcal{N}_{i}} {\textstyle \sum} w_{i j}\left(z_{i}-z_{j}+\lambda_i-\lambda_j\right) \bigr)
	\end{aligned}
	\end{align}
	\vspace{-0.6em}
\end{algorithm}

\smallskip
The agents exchange $\{ \bs{x}^i,z_i,\lambda_i \}$ with their neighbors only, therefore the controller can be implemented distributedly. In steady state, the agents should agree on their estimates, i.e., $\bs{x}^i=\bs{x}^j, \ \lambda_i=\lambda_j$, for all $i,j \in \mc{I}$. This motivates the presence of consensual terms for both the primal and dual variables. We denote  $\bs{E}_{q}:=\{\bs{y} \in \R^{N q}:\bs{y}=\1_{N} \otimes y, y\in \R^q\}$ the consensual subspace of dimension $q$, for some $q>0$, and $\bs{E}_{q}^\perp$ its orthogonal complement. Specifically, $\bs{E}_n$ is the estimate consensus subspace and $\bs{E}_m$ is the multiplier consensus subspace.  
To write the dynamics in compact form, let us define $\bs{x}=\col((\bs{x}^i)_{i\in\mc{I}})$, and, as in \cite[Eq.~13-14]{Pavel2018}, for all $i \in \mc{I}$,
\begin{subequations}
	\begin{align}
	\mathcal{R}_{i}:=&\left[ \begin{array}{lll}{{0}_{n_{i} \times n_{<i}}} & {I_{n_{i}}} & {\0_{n_{i} \times n_{>i}}}\end{array}\right], 
	\\
	\mathcal{S}_{i}:=&\left[ \begin{array}{ccc}{I_{n<i}} & {\0_{n<i \times n_{i}}} & {\0_{n<i \times n>i}} \\ {\0_{n>i \times n<i}} & {\0_{n>i \times n_{i}}} & {I_{n>i}}\end{array}\right],
	\end{align}
\end{subequations}
where $n_{<i}:=\sum_{j<i,j \in \mathcal{I}} n_{j}$, $n_{>i}:=\sum_{j>i, j \in \mathcal{I}} n_{j}$. In simple terms, $\mathcal R _i$ selects the $i$-th $n_i$ dimensional component from an $n$-dimensional vector, while $\mathcal S_i$ removes it. Thus, $\mathcal{R}_{i} \bs{x}^{i}=x_i$ and $\mathcal{S}_{i} \bs{x}^{i}=\bs{x}_{-i}^{i}$. We define $\mathcal{R}:=\operatorname{diag}\left((\mathcal{R}_{i})_{i \in \mathcal{I}}\right)$, $ \mathcal{S}:=\operatorname{diag}\left((\mathcal{S}_{i})_{i \in \mathcal{I}}\right)$. It follows that $x=\mathcal{R} \bs{x}$,  $\operatorname{col}((\bs{x}_{-i}^{i})_{i \in \mathcal{I}})=\mathcal{S}{\bs{x}} \in \R^{(N-1) n}$, and
$
	\bs{x}=\mathcal{R}^\top x+\mathcal{S}^\top\mc{S} \bs{x}.
$
Let $\bs \lambda:=\operatorname{col}((\lambda_i)_{i\in\mathcal I})$, $\boldsymbol{\Lambda}:=\operatorname{diag} \left( (A_{i})_{i \in \mathcal{I}}\right)$, 
$\boldsymbol{b}:=\operatorname{col}\left((b_{i})_{i \in \mathcal{I}}\right)$, $\bs{L}_{\textnormal{x}}:=L \otimes I_{n}$, $\bs{L}_{\lambda}:=L \otimes I_{m}$, $\boldsymbol{z}:=\operatorname{col}\left((z_{i})_{i \in \mathcal{I}}\right)$.  Further, we define the \emph{extended pseudo-gradient} mapping  $\bs{F}$ as:
\begin{align}
\label{eq:extended_pseudo-gradient}
\bs{F}(\bs{x}):=\operatorname{col}\left((\nabla_{\!x_{i}} J_{i}\left(x_{i}, \bs{x}_{-i}^{i}\right))_{i \in \mathcal{I}}\right).
\end{align}
The overall closed-loop system, in compact form, reads as:
\begin{subequations}
	\label{eq:dynamics:CG}
	\begin{align}\label{eq:dynamics:CG:a}
	\nonumber \dot { \bs{x}}=&
	\mathcal{R}^\top  \Pi_{\Omega} \bigl(\mathcal{R} \bs{x}, -\bigl(\bs{F} (\bs{x})+\boldsymbol{\Lambda}^\top \boldsymbol{\lambda}+c \mathcal{R} \bs{L}_{\textnormal{x}} \bs{x}\bigr)\bigr) +
	\\
	&
	\mathcal{S}^\top\left(-c \, \mathcal{S} \bs{L}_{\textnormal{x}} \bs{x} \right)
	\\
	\label{eq:dynamics:CG:b}
	\dot{\boldsymbol{z}}= & 
	\bs{L}_{\lambda} \boldsymbol{\lambda}
	\\
	\label{eq:dynamics:CG:c}
	\dot{\boldsymbol{\lambda}}= & \Pi_{\R_{\geq 0}^{N m}}\left(\boldsymbol{\lambda},\left(\boldsymbol{\Lambda}\mathcal R \bs{x}-\boldsymbol{b}\right.\right.  {-\bs{L}_{\lambda} \boldsymbol{\lambda}-\bs{L}_{\lambda} \boldsymbol{z} ) )}.
	\end{align}
\end{subequations}
The following Lemma relates the equilibria of the system in \eqref{eq:dynamics:CG} to the \gls{v-GNE} of the game in \eqref{eq:game}.

\noindent The proof  is analogous to \cite[Th.~1]{Pavel2018}, hence it is omitted. 
\smallskip

\begin{lemma}\label{lem:equilibria-vGNE}
The following statements hold:
	\begin{enumerate}
		\item[i)] Any equilibrium point $\bar{\boldsymbol{\omega}}=\operatorname{col}\left( \bar{\bs{x}},\bar{\bs z},\bar{\bs \lambda}\right) $ of the dynamics in  \eqref{eq:dynamics:CG} is such that $\bar{\bs{x}}=\1_{N} \otimes x^{*}$, $\bar{\bs \lambda}=\1_{N} \otimes \lambda^{*}$, where the pair $(x^*,\lambda^*) $ satisfies the KKT conditions in \eqref{eq:KKT}, hence $x^*$ is the \gls{v-GNE} of the game in \eqref{eq:game};
		\item[ii)] 
The set of equilibrium points of \eqref{eq:dynamics:CG} is nonempty. 	{\hfill $\square$}
	\end{enumerate}    
\end{lemma} 

\smallskip
\begin{remark}\label{rem:extpseudo}
	When considering Algorithm~\ref{algo:1} in absence of coupling constraints, we retrieve the controller in \cite[Eq.~47]{GadjovPavel2018}.
	In Algorithm~\ref{algo:1}, each agent evaluates the gradient of its cost function in its local estimate, not on the actual collective strategy. 
	In fact, only when the estimates belong to the consensus subspace, i.e., $\bs{x}=\1\otimes x$ (i.e., the estimate $\bs{x}^i$ of each agent coincide with the real actions $x$,  for example in the case of full-information), we have that $\bs{F}(\bs{x})=F(x)$.
	It follows that the operator $\mc{R}^\top\bs{F}$ is not necessarily monotone, not even if the pseudo gradient $F$ in \eqref{eq:pseudo-gradient} is strongly monotone, as in  Standing Assumption~\ref{Ass:StrMon}. This is the main technical difficulty that arises in (G)\gls{NE} problems under partial-information. 	{\hfill $\square$}
\end{remark}

\smallskip
\begin{lemma}\label{lem:LipschitzExtPseudo}
	 The extended pseudo-gradient mapping $\bs{F}$ in \eqref{eq:extended_pseudo-gradient} is $\theta$-Lipschitz continuous, for some $\mu\leq\theta\leq \theta_0$: for any  $\bs{x},\bs{y}\in \R^{Nn}$, $\| \bs{F} (\bs{x})-\bs{F}(\bs{y})\|\leq \theta \|\bs{x}-\bs{y}\|$.
	{\hfill $\square$} \end{lemma}
\begin{proof} See Appendix~\ref{app:lem:LipschitzExtPseudo}.
	\end{proof}
%

\smallskip
Under Lipschitz continuity of $\bs{F}$,
the work   \cite{Pavel2018} showed the following \emph{restricted strong monotonicity} property, which is crucial to prove convergence of the dynamics in \eqref{eq:dynamics:CG}.

\smallskip
\begin{lemma}[{\cite[Lem. $3$]{Pavel2018}}] \label{lem:strongmon_constant}
	Let
	\begin{align}
\label{eq:M1}
\mc{M}:=\begin{bmatrix}{\frac{\mu}{N}} & \ {-\frac{\theta_0+\theta}{2\sqrt{N}}} \\ {-\frac{\theta_0+\theta}{2\sqrt{N}}} & \ {c \lambda_{2}(L)-\theta}\end{bmatrix}, 
\quad \underline{c} :=\textstyle \frac{(\theta_0+\theta)^{2}+4\mu\theta}{4\mu\lambda_2(L)}.
\end{align}
	For any $c>\underline{c}$, for any $\bs{x}$ and any $\bs{x}^{\prime}\in \bs{E}_{n}$,  it holds that $\mc{M} \succ 0$  and that
    \[
	\begin{multlined}[b]
	\left(\bs{x}-\bs{x}^{\prime}\right)^\top\left(\mathcal{R}^\top\left(\bs{F}(\bs{x})-\bs{F}\left(\bs{x}^{\prime}\right)\right)+c \bs{L}_{\textnormal{x}}\left(\bs{x}-\bs{x}^{\prime}\right)\right) \\ \geq \lambda_{\textnormal{min}} (\mc{M})\left\|\bs{x}-\bs{x}^{\prime}\right\|^{2}.
	\end{multlined} \QEDopenhereeqn \]
\end{lemma}


\smallskip
By leveraging Lemma~\ref{lem:strongmon_constant}, we can  now prove the main result of this section, i.e., the convergence of the dynamics in  \eqref{eq:dynamics:CG} to a \gls{v-GNE}. 

\smallskip
\begin{theorem}\label{th:main1}
	Let $c>\underline{c}$, with  $\underline{c}$ as in \eqref{eq:M1}, and let  $\bs{\Omega}:=\{\bs{x} \in \R^{Nn} \mid \mc{R}\bs{x}\in \Omega\}$. For any initial condition in $\Xi=\bs{\Omega}\times \R^{mN}\times \R^{mN}_{\geq 0}$, the dynamics in \eqref{eq:dynamics:CG} have a unique Carathéodory solution, which belongs to $\Xi$ for all $t\geq 0$.  The solution converges to an equilibrium $\operatorname{col}\left( \bar{\bs{x}},\bar{\bs z},\bar{\bs \lambda}\right) $, with  $\bar{\bs{x}}=\1_{N} \otimes x^{*}$, $\bar{\bs \lambda}=\1_{N} \otimes \lambda^{*}$, where the pair $(x^*,\lambda^*)$ satisfies the KKT conditions in \eqref{eq:KKT}, hence $x^*$ is the \gls{v-GNE} of the game in \eqref{eq:game}.
	{\hfill $\square$}
\end{theorem}
\begin{proof}
	See Appendix~\ref{app:th:main1}.
\end{proof}
\smallskip

\begin{remark}
In Algorithm~\ref{algo:1}, each agent keeps and exchanges an estimate of the strategies of all other agents. Thus, the computation and communication costs increase with the number of agents. An open research direction is to design dynamics that allow each agent to estimate the strategies of only some of its competitors, when the inference graph is sparse (i.e., when the cost of each agent only depends on the action of a limited subset of other agents). \hfill $\square$
\end{remark}



\section{Double-integrator agents}   \label{sec:multiintegrators} 
\noindent In this section, we make the following additional assumption.

\smallskip
\begin{assumption}[{\cite[Ass.~1]{RomanoPavel2019}}] \label{Ass:Unboundedfeasibleset}
$\Omega=\R^n.$
	{\hfill $\square$} \end{assumption}
\smallskip
Moreover, we model each agent as a double integrator:
\begin{subequations}
	\label{eq:doubleint}
	\begin{empheq}[left={\forall i\in \mc{I}:\qquad\empheqlbrace }]{align}
	\dot{x}_{i}=&v_{i}
	\\
	\dot{v}_{i}=&u_{i}, \qquad \qquad 
	\end{empheq}
\end{subequations}
where $(x_i, v_i)$ is the state of agent $i$, and $u_i\in \R^{n_i}$ its control input. Our objective is to drive the agents' actions, i.e., the $x_i$ coordinates of their state, to a \gls{v-GNE} of the game in \eqref{eq:game}. We emphasize that in \eqref{eq:doubleint} we cannot directly control the agent's action. Moreover, at steady state, the velocities $v_i$ of all the agents must be zero. This scenario has been considered recently in \cite{RomanoPavel2019}, for games \emph{without} coupling constraints. 

In \eqref{eq:doubleint}, we consider the input $u_i=\frac{1}{h_i} (\tilde{u}_i -v_i)$, where $h_i>0$ is a positive scalar and $\tilde{u}_i$ has to be chosen appropriately, for all $i\in\mc{I}$; moreover, as in \cite{RomanoPavel2019}, let us define the coordinates transformation
\begin{align}\label{eq:defzeta}
\zeta_i:=x_i+h_i v_i.
\end{align}
%
%
The quantity $\zeta_i$ can be interpreted as a prediction of the position of agent $i$, given a forward step $h_i$. The closed-loop system in the new coordinates then  reads as
\begin{subequations}
	\label{eq:dynamicsdoubletranslated}
	\begin{empheq}[left={\forall i\in \mc{I}:\qquad\empheqlbrace }]{align}
	\label{eq:dynamics40}
	\dot{v}_{i}=&\textstyle \frac{1}{h_i}(\tilde{u}_i-v_i) \qquad
	\\
	\label{eq:dynamics41}
	\dot{\zeta}_{i}=&\tilde{u}_{i}.
	\end{empheq}
\end{subequations}
We note that the dynamics of the variable $\zeta_i$ in \eqref{eq:dynamics41}, under Assumption~\ref{Ass:Unboundedfeasibleset}, are identical to the single-integrator in  \eqref{eq:integrators}, with translated input $\tilde{u}_i$. As such, we are are able to design the input $\tilde{u}_i$, according 
to Algorithm~\ref{algo:1}, to drive $\zeta:=\operatorname{col}((\zeta_i)_{i\in\mc{I}})$ to an equilibrium $\bar{\zeta}=x^*$, where $x^*$ is the \gls{v-GNE} for the game in \eqref{eq:game}.
Moreover, the velocity dynamics \eqref{eq:dynamics40} are \gls{ISS} with respect to the input $u_i$ \cite[Lemma~4.6]{Khalil}.
Finally, we remark 
that, at any equilibrium of \eqref{eq:doubleint}, $v_i=\0_{n_i}$, hence $\zeta_i=x_i$, for all $i \in \mc I$.
Building on this considerations, we propose Algorithm~\ref{algo:di-constantgain} to drive the double-integrator agents \eqref{eq:doubleint} towards a \gls{v-GNE}.

 Differently from Algorithm~\ref{algo:1}, the agents are not keeping an estimate of other agents' actions, but of other agents \emph{predictions}.
Here, $\bs{ \zeta}^i=\operatorname{col}((\bs{\zeta}_{j}^{i})_{j \in \mc I})$, and $\bs{\zeta}_{j}^{i}$ represents agent $i$'s estimation of the quantity $\zeta_j=x_j+h_jv_j$ for $j\neq i$, while $\bs{\zeta}_i^i:=x_i+h_iv_i=\zeta_i$. Let us denote $\bs{\zeta}=\col((\bs{\zeta}^i)_{i\in\mc{I}})$, $\bs{\zeta}^j_{-i}=\col((\bs{\zeta} ^j_\ell)_{\ell\in\mc{I}\backslash\{i\}})$, $H=\diag((h_i I_{n_i})_{i\in\mc{I}})$. \vspace{0.8em}
\begin{algorithm}[b] \caption{Distributed GNE  seeking (double-integrators)}\label{algo:di-constantgain}
	\vspace{-0.4em}
	\begin{align*}
	\begin{aligned}
	\dot{x}_{i}=&v_{i}, \qquad 
	\dot{v}_{i}=u_{i}
	\\
	u_i=&\textstyle -\frac{1}{h_i}\bigl(\nabla_{\! x_i} J_{i}(\bs{\zeta}^{i}_{i}, \bs{\zeta}^{i}_{-i})+A_{i}^{\top} \lambda_{i}
	\\ &
	\qquad \quad +\textstyle c\sum _{j \in \mathcal{N}_{i}} w_{i j}(\bs{\zeta}^i_i-\bs{\zeta}^j_i)  \bigr)
	-\frac{1}{h_i}v_i 
	\\
	\dot{\bs{ \zeta}}^i_{-i}=&- \textstyle\sum _{j \in \mathcal{N}_{i}} w_{i j}( \bs{\zeta}^i_{-i}-\bs{\zeta}^j_{-i} ), \quad \bs{\zeta}^{i}_i=x_i+h_iv_i
	\\
	{\dot z}_i=&\textstyle \sum_{j \in \mathcal{N}_{i}} w_{i j}(\lambda_{i}-\lambda_{j}) 
	\\
	\dot\lambda_i=&\Pi_{\R_{\geq 0}^m}\bigl(\lambda_i, A_{i} \bs{\zeta}^{i}_{i}-b_{i}-
	\underset{{j \in \mathcal{N}_{i}}}{\textstyle\sum} w_{i j} (z_{i}-z_{j}+\lambda_i-\lambda_j)\bigr)
	\end{aligned}
	\end{align*}
	\vspace{-0.6em}
\end{algorithm}

In compact form, the closed-loop system reads as
\begin{subequations}
	\label{eq:dynamics:CG:double}
	\begin{align}\label{eq:dynamics:CG:double:a} 
	\dot x=& v 
	\\
	\label{eq:dynamics:CG:double:b} 
	H\dot v=&-(\bs{F} (\bs{\zeta})+\boldsymbol{\Lambda}^\top \boldsymbol{\lambda}+ c\mathcal{R}\bs{L}_{\textnormal{x}} \bs{\zeta}) -v 
	\\
	\label{eq:dynamics:CG:double:c}
	\mathcal{S}\dot { \bs{\zeta}}=&
	-c\mathcal{S}\bs{L}_{\textnormal{x}} \bs{\zeta}, \quad  \mathcal{R} { \bs{\zeta}}=x+Hv
	\\ 
	\label{eq:dynamics:CG:double:d}
	\dot{\boldsymbol{z}}=& 
	\bs{L}_{\lambda} \boldsymbol{\lambda} 
	\\
	\label{eq:dynamics:CG:double:e}
	\dot{\boldsymbol{\lambda}}= & \Pi_{\R_{\geq 0}^{N m}}\left(\boldsymbol{\lambda},\left(\boldsymbol{\Lambda}\mathcal R \bs{\zeta}-\boldsymbol{b}\right.\right.  {-\bs{L}_{\lambda} \boldsymbol{\lambda}-\bs{L}_{\lambda} \boldsymbol{z} ) )}.
	\end{align}
\end{subequations}
\begin{theorem}\label{th:doubleint}
  Let Assumption~\ref{Ass:Unboundedfeasibleset} hold. For any initial condition with $\bs{\lambda}(0)\in\R^{mN}_{\geq 0}$, the equations in \eqref{eq:dynamics:CG:double} have a unique Carathéodory solution, such that $\bs \lambda(t)\in \R^{Nm}_{\geq 0}$, for all $t\geq 0$.  The solution converges to an equilibrium $\operatorname{col}\left( \bar{x}, \bar{v}, \bar{\bs{\zeta}},\bar{\bs z},\bar{\bs \lambda}\right) $, with  $\bar x=x^*, \bar{v}=\0_{n}$, $\bar{\bs{\zeta}}=\1_{N} \otimes x^{*}$, $\bar{\bs \lambda}=\1_{N} \otimes \lambda^{*}$, where the pair $(x^*,\lambda^*)$ satisfies the KKT conditions in \eqref{eq:KKT}, so $x^*$ is a \gls{v-GNE} for the game in \eqref{eq:game}.
	{\hfill $\square$}
\end{theorem}
\begin{proof}
See Appendix~\ref{app:th:doubleint}.
\end{proof}

\smallskip
Algorithm~\ref{algo:di-constantgain} is derived by choosing $\tilde{u}_i$ in \eqref{eq:dynamicsdoubletranslated} according to Algorithm~\ref{algo:1}. The proof of Theorem~\ref{th:doubleint} is not based on the specific structure of Algorithm~\ref{algo:1}, but only on its convergence properties, hence the result still holds if another controller with similar features is selected in place of Algorithm~\ref{algo:1}.
%
In \cite{RomanoPavel2019}, the authors addressed \gls{NE} problems  and chose the inputs $\tilde{u}_i$ according to the algorithm presented in \cite[Eq.~47]{GadjovPavel2018}. The controller in \cite{GadjovPavel2018}  achieves exponential convergence to a \gls{NE}, hence \gls{ISS} with respect to possible additive disturbances \cite[Lemma~4.6]{Khalil}. Therefore, in \cite{RomanoPavel2019}, the authors were able
to tackle the presence of deterministic disturbances, via an asymptotic observer and by leveraging \gls{ISS} arguments.
We have not guaranteed this robustness, i.e., exponential convergence,  for the primal-dual dynamics in \eqref{eq:dynamics:CG}.
However, the controller in \cite{RomanoPavel2019} is designed for games without any constraints (local or shared). 
On the contrary, the controller in Algorithm~\ref{algo:di-constantgain} drives the system in \eqref{eq:doubleint} to a \gls{v-GNE} of a generalized game, and ensures for the coupling constraints to be satisfied asymptotically. 
Also, like in \cite{RomanoPavel2019},  we assumed
the absence of local constraints  (Assumption~\ref{Ass:Unboundedfeasibleset}).
Nevertheless, if some are present, they can be included in the coupling constraints, hence dualized and satisfied asymptotically.

\section{Numerical example: mobile sensor network}\label{sec:simulations}

We consider a numerical example, inspired by connectivity control problems for  sensor networks \cite{Stankovic_Johansson_Stipanovic_2012}, \cite{RomanoPavel2019}. Each of five agent is represented by a robot/vehicle, moving in a plane, designed to optimize some private primary objective related to its position, provided that overall connectivity is preserved over the network.  For each agent $i\in\mc{I}=\{1,\dots,5\}$, its cost function is
$ \textstyle J_{i}\left(p_{i}, p_{-i}\right):=p_{i}^{\top} p_{i}+p_{i}^{\top} r_{i}+\sum_{j \in \mathcal{I}}\left\|p_{i}-p_{j}\right\|^{2},$ 
with $p_i=\operatorname{col}(x_i,y_i)$ its cartesian coordinates, $r_i\in\R^2$ a local parameter.
We assume the local constraints $0.1\leq y_i \leq 0.5$, $\forall i\in\mc{I}$. In order for all the agents to maintain communication with their neighbors, we impose the Chebyschev distance between any two neighboring robots to be smaller than $0.2 \,m$.  Hence the (affine) coupling constraints are represented by $\max\{|x_i-x_j|,|y_i-y_j|\}\leq0.2, \forall (i,j)\in \mc{E}$. As common for autonomous vehicles, we model the agents as single- or double-integrators.\newline \indent
\emph{Velocity-actuated  robots}: each agent is modeled as in \eqref{eq:integrators} and we apply Algorithm~\ref{algo:1}.
\newline \indent
 \emph{ Force-actuated robots }: Each agent has a dynamic as in \eqref{eq:doubleint}, under Algorithm~\ref{algo:di-constantgain}. The local constraints are dualized and will be satisfied asymptotically  (see Section~\ref{sec:multiintegrators}).
\newline\indent
The initial conditions are chosen randomly and we fix $c=30$ to satisfy the condition in Theorem~\ref{th:main1}.
Figure~\ref{fig:2} illustrates the results for the two cases and shows convergence the \gls{GNE} of the game and asymptotic satisfaction of the coupling constraints.
Finally, in Figure~\ref{fig:3}, we compare the trajectories of the five robots in the velocity-actuated and force-actuated scenario. In the two cases, the agents are converging to the same, unique \gls{v-GNE}. However, the local constraints are satisfied along the whole trajectory for single integrator agents, only asymptotically for the double integrator agents.
\setlength{\textfloatsep}{\textfloatsepsave-0.5em}
\begin{figure}[t]
	\centering
	\includegraphics[width=0.9\columnwidth]{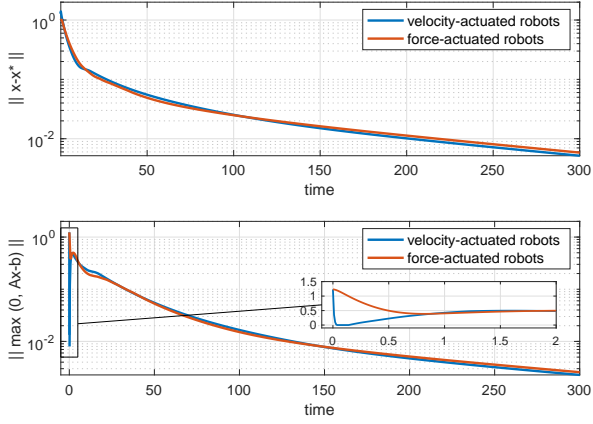}
	\caption{Results of Alg.~\ref{algo:1} for single- and double-integrator agents.}
	\label{fig:2}
\end{figure}
\begin{figure}[t]
	\centering
	\includegraphics[width=\columnwidth]{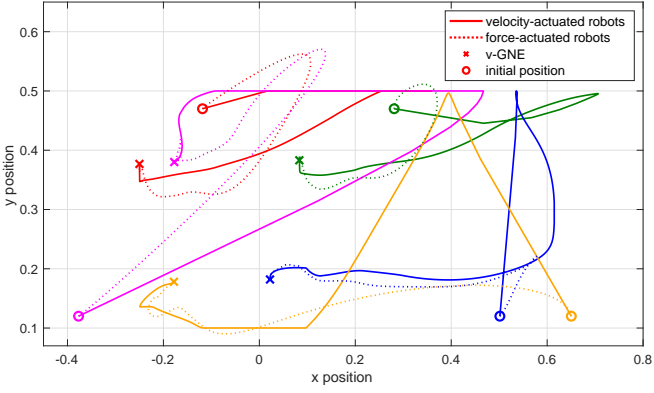}
	\caption{Velocity-actuated and force-actuated robots trajectories.}
	\label{fig:3}
\end{figure}
\section{Conclusion and outlook}\label{sec:conclusion}
Generalized games played by double-integrator agents can be solved via a fully distributed primal-dual projected-pseudogradient dynamic controller, if the game mapping is strongly monotone and Lipschitz continuous. Seeking an equilibrium in games with compact action sets or constrained dynamics is currently an unexplored problem. The extension of our results to networks of heterogeneous dynamical systems is left as future research.
\section*{Appendix}
 \subsection{Proof of Lemma~\ref{lem:LipschitzExtPseudo}}\label{app:lem:LipschitzExtPseudo}
 \noindent 
Let us define $\bs{x}=\col( (\bs{x}^i )_{i\in \mc{I}})$,
$\bs{y}=\col((\bs{y}^{i})_{i\in \mc{I}})$.
By Standing Assumption~\ref{Ass:StrMon}, we have, for all $i\in\mc{I}$, 
$
\| \nabla_{\! x_i} J_{i}(\bs{x}^i)-\nabla_{\!x_i}J_{i}(\bs{y}^{i})\| \leq \| F(\bs{x}^i)-F(\bs y^i)\|\leq  \theta_{0} \| \bs{x}^i-\bs{y}^{i} \|.
$
Therefore it holds that
\begin{align*}
\|\bs{F}(\bs{x})-\bs{F}(\bs{y})\|^2&=\textstyle \sum_{i\in\mc{I}} \| \nabla_{\! x_i} J_{i}(\bs{x}^i)-\nabla_{\! x_i}J_{i}(\bs{y}^{i})\|^2
\\ &
\leq \theta_{0}^2  \textstyle\sum_{i\in\mc{I}}  \| \bs{x}^i-\bs{y}^{i} \|^2=\theta_{0}^2 \|\bs{x}-\bs{y}\|^2.
\end{align*}
That $\theta\geq \mu$ follows by choosing $\mc{S}\bs{x}=\mc{S}\bs{y}$, $\bs{x}\neq\bs{y}$.
\hfill  \QED
\subsection{Proof of Theorem~\ref{th:main1}}\label{app:th:main1}
Under Standing Assumption~\ref{Ass:Graph}, we have, for any $q>0$,
\begin{align}\label{eq:null_Lq}
\operatorname{N u l l}\left(L\otimes I_q\right)= \operatorname{Range}\left(\1_{N} \otimes I_{q}\right)=\bs{E}_{q}.
\end{align}
We first rewrite the dynamics as
\begin{align}
\label{eq:compact_constant}
\dot{\boldsymbol{\omega}}=\Pi_{\Xi} ({\boldsymbol{\omega}},-\mathcal{B}(\bs \omega)-\Phi \bs{\omega}),
\end{align}
where  ${\boldsymbol{\omega}}=\operatorname{col}\left( {\bs{x}},{\bs z},{\bs \lambda}\right) $,
\begin{align*}
\Phi= \left[ 
\begin{smallmatrix}
{0} & {0} & {\mathcal{R}^\top \boldsymbol{\Lambda}^\top} \\ {0} & {0} & {-\bs{L}_{\lambda}} \\ {-\boldsymbol{\Lambda} \mathcal{R}} & {\bs{L}_{\lambda}} & {0}
\end{smallmatrix}\right], \quad
\mathcal{B}(\bs \omega)=
\left[
\begin{smallmatrix}
\mathcal{R}^\top \bs F(\bs{x})+c\bs{L}_{\textnormal{x}} \bs{x}\\ 
\bs 0_{Nm}\\
\bs{L}_\lambda \bs \lambda +\bs b
\end{smallmatrix}\right].
\end{align*}
$\mathcal{B}$ is Lipschitz continuous by Standing Assumption~\ref{Ass:StrMon} and $\Xi$ is closed and convex by Standing Assumption~\ref{Ass:Convexity}. We conclude that that there exists a unique Carathéodory solution to \eqref{eq:compact_constant}, that belongs to $\Xi$ for every $t\geq0$,
Consider the quadratic Lyapunov function 
\begin{align*}
V(\bs{\omega})&=\textstyle\frac{1}{2}\|\bs{\omega}-\bar{\bs\omega} \|^2,
\end{align*}
where $\bar{\bs{\omega}}:=\col(\bar{\bs{x}},\bar{\bs{z}}, \bar{\bs{\lambda}} ) $ is any equilibrium of  \eqref{eq:dynamics:CG}. We remark that, by Lemma~\ref{lem:equilibria-vGNE}, an equilibrium exists, and  $\bar{\bs{x}}=\1_{N} \otimes x^{*}$, $\bar{\bs \lambda}=\1_{N} \otimes \lambda^{*}$, with  $(x^*,\lambda^*)$ satisfying the KKT conditions in \eqref{eq:KKT}. We can apply Lemma~\ref{lem:minoratingproj} to obtain
\begin{align}
\nonumber
\dot V(\bs \omega ):=&\nabla V(\boldsymbol{\omega}) \dot{\boldsymbol{\omega}}= 
\nonumber
 (\boldsymbol{\omega}-\bar{\boldsymbol{\omega}})^{\top} \Pi_{\Xi} ({\boldsymbol{\omega}},-\mathcal{B}(\bs \omega)-\Phi \bs{\omega})
\\
\label{eq:usefulstep}
\leq  & (\boldsymbol{\omega}-\bar{\boldsymbol{\omega}})^{\top} (-\mathcal{B}(\bs \omega)-\Phi \bs{\omega}).
\end{align}
By Lemma~\ref{lem:minoratingproj}, it also holds that
$
(\boldsymbol{\omega}-\bar{\boldsymbol{\omega}})^{\top} (-\mathcal{B}(\bar{\bs \omega})-\Phi \bar{\bs{\omega}})\leq 0.
$
By subtracting this term from \eqref{eq:usefulstep}, we obtain
\begin{align}
\nonumber
\dot V&\leq  -(\bs \omega - \bs{\bar \omega })^\top \left( \mathcal{B}(\bs \omega ) -\mathcal{B}(\bs{ \bar  \omega} ) +\Phi(\bs \omega -\bs{\bar \omega })\right)
\\
\label{eq:usefulstep2}
&  \begin{aligned}
=&-(\bs{x}-\bar{\bs{x}})^{\top} \mathcal{R}^{\top}
\left( \bs{F}(\bs{x})-\bs{F}(\boldsymbol{\bar x}) \right) 
\\
&-(\bs{x}-\bar{\bs{x}})^{\top} c\bs{L}_{\textnormal{x}} (\bs{x}-\bar{\bs{x}})
 -(\bs \lambda -\bar{\bs \lambda})^\top \bs{L}_{\lambda} (\bs \lambda -\bar{\bs \lambda}),
\end{aligned}
\end{align}
where, in the last equality, we used that, $\Phi^\top=-\Phi$.  
By \eqref{eq:null_Lq} and \cite[Cor.~18.16]{Bauschke2017}, we have that
$(\bs \lambda -\bar{\bs \lambda})^\top \bs{L}_{\lambda} (\bs \lambda -\bar{\bs \lambda})\geq \frac{1}{2\lambda_{\textnormal{max}}(L)} \| \bs{L}_{\lambda} \boldsymbol{\lambda} \|^2 $.
Finally, by Lemma~\ref{lem:strongmon_constant}, we obtain
\begin{align}
\label{eq:upperbound_constant}
\dot V\leq
- \lambda_{\textnormal{min}}(\mc{M}) \|  \bs{x}-{\bar {\bs{x}}\|^2-\textstyle \frac{1}{2\lambda_{\textnormal{max}}(L)}\| \bs{L}_\lambda \bs \lambda \|^2 }\leq 0,
\end{align}
with $\mc{M}\succ 0$ as in \eqref{eq:M1}.
By noticing that $V$ is radially unbounded, we conclude that the solution to \eqref{eq:dynamics:CG} is bounded. Besides,  by  \cite[Th.~2]{DePersisGrammatico2018}, the solution converges to the largest invariant set $\mc{O}$ contained in $\mc{Z}:=\{\bs \omega \text{ s.t. }\dot V(\bs \omega)=0\}$. \newline \indent
We first characterize any point $ \operatorname{col}( \hat{\bs{x}},\hat{\bs z},\hat{\bs \lambda}) \in \mc{Z} $, for which the quantities in \eqref{eq:usefulstep}-\eqref{eq:upperbound_constant} must be zero. By \eqref{eq:upperbound_constant}, $\hat{\bs{x}}=\bar{\bs{x}}=\1_N\otimes x^*,$ and $\hat{\bs \lambda}\in \bs{E}_{m}$, i.e.  $\hat{\bs \lambda}=\1_N \otimes \hat\lambda$, for some  $\hat\lambda \in \R^m_{\geq 0}$. Also, by expanding \eqref{eq:usefulstep}, and by  $\hat{\bs{x}}=\bar{\bs{x}}$, $\bs{L}_\lambda \bs{\hat\lambda}=0$, we have
\begin{align}
\nonumber 
0&=(\bs{\hat\lambda}-\bs{\bar \lambda})^\top (\boldsymbol{\Lambda}\mathcal R\bar{\bs{x}}-\boldsymbol{b} {-\bs{L}_{\lambda} \boldsymbol{\hat z} } )
=(\hat\lambda-\lambda^*)^\top (Ax^*-b) 
\\
\label{eq:usefulstep3}
&=\hat\lambda^\top (Ax^*-b)
=\bs{\hat\lambda}^\top (\boldsymbol{\Lambda}\mathcal R\bar{\bs{x}}-\boldsymbol{b} {-\bs{L}_{\lambda} \boldsymbol{\hat \lambda}-\bs{L}_{\lambda} \boldsymbol{\hat z} } ),
\end{align}
where in the second equality we have used \eqref{eq:null_Lq} and the fourth equality follows from the  KKT conditions in \eqref{eq:KKT}. This concludes the characterization of the set $\mc Z$. 
\newline \indent
By invariance, any trajectory $\underline{\bs{\omega}}(t)=\operatorname{col}( \underline{\bs{x}}(t),\underline{\bs z}(t),\underline{\bs \lambda}(t))$ starting at any $\operatorname{col}( \underline{\bs{x}}, \underline{\bs z},\underline{\bs \lambda}) \in \mc{O}$ must lie in $\mc{Z} \supset \mc{O}$ for all $t\geq 0$. Therefore, $\underline{\bs{x}}(t)\equiv\bar{\bs{x}}$ and $\underline{\bs \lambda}(t)\in \bs{E}_{m}$ for all $t$. Moreover, $\dot{\underline{\bs z}}(t)=0$, for all $t$, by \eqref{eq:dynamics:CG:b}, or ${\underline{\bs z}}(t)\equiv{\underline{\bs z}}$. Hence the quantity $v:= \left(\boldsymbol{\Lambda}\mathcal R\underline{\bs{x}}(t)-\boldsymbol{b} {-\bs{L}_{\lambda} \underline{\boldsymbol{\lambda}}(t)-\bs{L}_{\lambda} \underline{\boldsymbol{z}}(t) } \right)$ is constant along the trajectory $\underline{\bs{\omega}}$.
Suppose by contradiction that $v_k>0$, where $v_k$ denotes the $k$-th component of $v$. Then, by \eqref{eq:dynamics:CG:c}, $\dot{\underline{\bs \lambda}}(t)_k=v_k$ for all $t$, and ${\underline{\bs \lambda}}(t)$ grows indefinitely. Since all the solutions of \eqref{eq:dynamics:CG} are bounded, this is a contradiction. Therefore, $v\leq0$, and $\underline{\bs \lambda}(t)^\top v=0$ by \eqref{eq:usefulstep3}. Equivalently, $\textstyle v\in \nc_{\R_{\geq 0}^{N m}}(\underline{\bs \lambda}(t))$, hence $\dot{\underline{\bs \lambda}}(t)=0$, for all $t$. We conclude that all the points in the set $\mc{O}$ are equilibria.

%
%
%
%

The set $\Lambda (\bs{\omega}_0)$ of $\omega$-limit points\footnote{$z :[0, \infty) \rightarrow \mathbb{R}^{n}$ has an $\omega$-limit point at $\bar{z}$ if there exists a nonnegative diverging sequence $\{t_k\}_{k\in\N}$  such that $ z\left(t_{k}\right) \rightarrow \bar{z}$.}of the solution to \eqref{eq:dynamics:CG} starting from any $\bs{\omega}_0\in\Xi$ is nonempty (by Bolzano-Weierstrass theorem, since all the trajectories of \eqref{eq:dynamics:CG} are bounded) and  invariant (as in proof of \cite[Lemma~5]{DePersisGrammatico2018}). By $\dot{V}\leq 0$  it follows that $V$ must be constant on $\Lambda(\bs{\omega}_0)$, hence $\Lambda(\bs{\omega}_0)\subseteq \mc{Z}$ (see proof of \cite[Th.2]{DePersisGrammatico2018}). Also $\Lambda(\bs{\omega}_0)$ is invariant, so $\Lambda(\bs{\omega}_0)\subseteq \mc{O}$. 
Since the distance to any equilibrium point along any trajectory of \eqref{eq:dynamics:CG} is non-increasing by \eqref{eq:upperbound_constant}, it follows that if a solution of  \eqref{eq:dynamics:CG}  has an $\omega$-limit point at an equilibrium, then the solution converges to that equilibrium. 
\hfill \QED

\subsection{Proof of Theorem~\ref{th:doubleint}}\label{app:th:doubleint}
	By applying the coordinate transformation $x\mapsto \mathcal{R} \bs{\zeta}=x+Hv$ to the system in \eqref{eq:dynamics:CG:double}, we obtain:
\begin{subequations}
	\label{eq:di_translated_compact} 
	\begin{align}
	\label{eq:di_translated_compact:a}
	 \dot v=&-H^{-1}(\bs{F} (\bs{\zeta})+\boldsymbol{\Lambda}^\top \boldsymbol{\lambda}+ c\mathcal{R}\bs{L}_{\textnormal{x}} \bs{\zeta}) -H^{-1}v 
	\\
	\dot { \bs{\zeta}}=&
	-\mathcal{R}^\top (\bs{F} (\bs{\zeta})+\boldsymbol{\Lambda}^\top \boldsymbol{\lambda}+ c\mathcal{R}\bs{L}_{\textnormal{x}} \bs{\zeta})
	\label{eq:di_translated_compact:b}
	 -c\mathcal{S}^\top\mathcal{S}\bs{L}_{\textnormal{x}} \bs{\zeta}  
	\\
	\label{eq:di_translated_compact:c}
	\dot{\boldsymbol{z}}=& 
	\bs{L}_{\lambda} \boldsymbol{\lambda} 
	\\
	\label{eq:di_translated_compact:d}
	\dot{\boldsymbol{\lambda}}= & \Pi_{\R_{\geq 0}^{N m}}\left(\boldsymbol{\lambda},\left(\boldsymbol{\Lambda}\mathcal R \bs{\zeta}-\boldsymbol{b}\right.\right.  {-\bs{L}_{\lambda} \boldsymbol{\lambda}-\bs{L}_{\lambda} \boldsymbol{z} ) )}.
	\end{align}
\end{subequations}
The system \eqref{eq:di_translated_compact} is in cascade form for \eqref{eq:di_translated_compact:a} with respect to \eqref{eq:di_translated_compact:b}-\eqref{eq:dynamics:CG:double:d}. 
Notice also that, under Assumption~\ref{Ass:Unboundedfeasibleset}, the subsystem \eqref{eq:di_translated_compact:b}-\eqref{eq:dynamics:CG:double:d} is exactly \eqref{eq:dynamics:CG}. Hence, there exists a unique solution to \eqref{eq:di_translated_compact:b}-\eqref{eq:dynamics:CG:double:d}, that is bounded and converges to an equilibrium point $\operatorname{col}\left(\1_{N} \otimes x^{*},\bar{\bs z},\1_{N} \otimes \lambda^{*}\right) $, where the pair $(x^*,\lambda^*)$ satisfies the KKT conditions in \eqref{eq:KKT}, by Theorem~\ref{th:main1}.  On the other hand, the dynamic \eqref{eq:di_translated_compact:a} is \gls{ISS} with respect to the input $\tilde u:=-H^{-1}(\bs{F} (\bs{\zeta})+\boldsymbol{\Lambda}^\top \boldsymbol{\lambda}+ c\mathcal{R}\bs{L}_{\textnormal{x}} \bs{\zeta})$ \cite[Lemma $4.6$]{Khalil}, and this input is bounded, by boundedness of the trajectory $\left( \bs \zeta,\bs{k}, \bs z, \bs \lambda\right)$ and Lemma~\ref{lem:LipschitzExtPseudo}. 
Moreover, since $ \bar{\bs{\zeta}}=\1_{N} \otimes x^{*} $, $\bar{\bs \lambda}=\1_{N} \otimes \lambda^{*}$, by the KKT conditions in \eqref{eq:KKT} and by continuity, we have $\tilde u\rightarrow \0_{n}$ for $t\rightarrow \infty$.
Therefore, $v(t)\rightarrow \0_{n}$ for $t\rightarrow \infty$  \cite[Ex.~4.58]{Khalil}. By definition of $\zeta_i=\mc{R}_i\bs{\zeta}^i$ in \eqref{eq:defzeta}, we can also conclude that $x\rightarrow x^*$.  \hfill  \QED





\bibliographystyle{IEEEtran}
\bibliography{library}

\begin{thebibliography}{10}
\providecommand{\url}[1]{#1}
\csname url@samestyle\endcsname
\providecommand{\newblock}{\relax}
\providecommand{\bibinfo}[2]{#2}
\providecommand{\BIBentrySTDinterwordspacing}{\spaceskip=0pt\relax}
\providecommand{\BIBentryALTinterwordstretchfactor}{4}
\providecommand{\BIBentryALTinterwordspacing}{\spaceskip=\fontdimen2\font plus
\BIBentryALTinterwordstretchfactor\fontdimen3\font minus
  \fontdimen4\font\relax}
\providecommand{\BIBforeignlanguage}[2]{{%
\expandafter\ifx\csname l@#1\endcsname\relax
\typeout{** WARNING: IEEEtran.bst: No hyphenation pattern has been}%
\typeout{** loaded for the language `#1'. Using the pattern for}%
\typeout{** the default language instead.}%
\else
\language=\csname l@#1\endcsname
\fi
#2}}
\providecommand{\BIBdecl}{\relax}
\BIBdecl

\bibitem{Saad2012}
W.~{Saad}, Z.~{Han}, H.~V. {Poor}, and T.~{Basar}, ``Game-theoretic methods for
  the smart grid: An overview of microgrid systems, demand-side management, and
  smart grid communications,'' \emph{IEEE Signal Processing Magazine}, vol.~29,
  no.~5, pp. 86--105, 2012.

\bibitem{Grammatico2017}
S.~{Grammatico}, ``Dynamic control of agents playing aggregative games with
  coupling constraints,'' \emph{IEEE Transactions on Automatic Control},
  vol.~62, no.~9, pp. 4537--4548, 2017.

\bibitem{Stankovic_Johansson_Stipanovic_2012}
M.~S. {Stankovic}, K.~H. {Johansson}, and D.~M. {Stipanovic}, ``Distributed
  seeking of {N}ash equilibria with applications to mobile sensor networks,''
  \emph{IEEE Transactions on Automatic Control}, vol.~57, no.~4, pp. 904--919,
  2012.

\bibitem{Palomar_Eldar_Facchinei_Pang_2009}
F.~Facchinei and J.~Pang, ``Nash equilibria: the variational approach,'' in
  \emph{Convex Optimization in Signal Processing and Communications}, D.~P.
  Palomar and Y.~C. Eldar, Eds.\hskip 1em plus 0.5em minus 0.4em\relax
  Cambridge University Press, 2009, p. 443–493.

\bibitem{DePersisMonshizadeh_2019}
C.~{De Persis} and N.~{Monshizadeh}, ``A feedback control algorithm to steer
  networks to a {C}ournot–{N}ash equilibrium,'' \emph{IEEE Transactions on
  Control of Network Systems}, vol.~6, no.~4, pp. 1486--1497, 2019.

\bibitem{GadjovPavel2018}
D.~{Gadjov} and L.~{Pavel}, ``A passivity-based approach to {N}ash equilibrium
  seeking over networks,'' \emph{IEEE Transactions on Automatic Control},
  vol.~64, no.~3, pp. 1077--1092, 2019.

\bibitem{DePersisGrammatico_TAC2020}
C.~{De Persis} and S.~{Grammatico}, ``Continuous-time integral dynamics for a
  class of aggregative games with coupling constraints,'' \emph{IEEE
  Transactions on Automatic Control, DOI: 10.1109/TAC.2019.2939639}, 2019.

\bibitem{YiPavel2019}
P.~Yi and L.~Pavel, ``An operator splitting approach for distributed
  generalized {N}ash equilibria computation,'' \emph{Automatica}, vol. 102, pp.
  111 -- 121, 2019.

\bibitem{Yu_VanderSchaar_Sayed_2017}
C.~{Yu}, M.~{van der Schaar}, and A.~H. {Sayed}, ``Distributed learning for
  stochastic generalized {N}ash equilibrium problems,'' \emph{IEEE Transactions
  on Signal Processing}, vol.~65, no.~15, pp. 3893--3908, 2017.

\bibitem{BelgioiosoGrammatico_ECC_2018}
G.~Belgioioso and S.~Grammatico, ``Projected-gradient algorithms for
  generalized equilibrium seeking in aggregative games are preconditioned
  forward-backward methods,'' \emph{2018 European Control Conference (ECC)},
  pp. 2188--2193, 2018.

\bibitem{FrihaufKrsticBasar_2012}
P.~{Frihauf}, M.~{Krstic}, and T.~{Basar}, ``Nash equilibrium seeking in
  noncooperative games,'' \emph{IEEE Transactions on Automatic Control},
  vol.~57, no.~5, pp. 1192--1207, 2012.

\bibitem{SalehisadaghianiWeiPavel2019}
F.~Salehisadaghiani, W.~Shi, and L.~Pavel, ``Distributed {N}ash equilibrium
  seeking under partial-decision information via the alternating direction
  method of multipliers,'' \emph{Automatica}, vol. 103, pp. 27 -- 35, 2019.

\bibitem{Koshal_Nedic_Shanbag_2016}
J.~Koshal, A.~{Nedić}, and U.~V. Shanbhag, ``Distributed algorithms for
  aggregative games on graphs,'' \emph{Operations Research}, vol.~64, pp.
  680--704, 2016.

\bibitem{DePersisGrammatico2018}
C.~{De Persis} and S.~Grammatico, ``Distributed averaging integral {N}ash
  equilibrium seeking on networks,'' \emph{Automatica, DOI:
  10.1016/j.automatica.2019.108548}, 2019.

\bibitem{Pavel2018}
L.~{Pavel}, ``Distributed {GNE} seeking under partial-decision information over
  networks via a doubly-augmented operator splitting approach,'' \emph{IEEE
  Transactions on Automatic Control, DOI: 10.1109/TAC.2019.2922953}, 2019.

\bibitem{DengNian2019}
Z.~{Deng} and X.~{Nian}, ``Distributed generalized {N}ash equilibrium seeking
  algorithm design for aggregative games over weight-balanced digraphs,''
  \emph{IEEE Transactions on Neural Networks and Learning Systems}, vol.~30,
  no.~3, pp. 695--706, 2019.

\bibitem{Zhang_et_al_NL_NE_2019}
Y.~{Zhang}, S.~{Liang}, X.~{Wang}, and H.~{Ji}, ``Distributed {N}ash
  equilibrium seeking for aggregative games with nonlinear dynamics under
  external disturbances,'' \emph{IEEE Transactions on Cybernetics}, pp. 1--10,
  2019.

\bibitem{RomanoPavel2019}
A.~R. {Romano} and L.~{Pavel}, ``Dynamic {N}ash equilibrium seeking for
  higher-order integrators in networks,'' in \emph{2019 18th European Control
  Conference (ECC)}, 2019, pp. 1029--1035.

\bibitem{Bauschke2017}
H.~H. Bauschke and P.~L. {Combettes}, \emph{Convex analysis and monotone
  operator theory in Hilbert spaces}.\hskip 1em plus 0.5em minus 0.4em\relax
  Springer, 2017, vol. 2011.

\bibitem{FacchineiKanzow2010}
F.~Facchinei and C.~Kanzow, ``Generalized {N}ash equilibrium problems,''
  \emph{Annals of Operations Research}, vol. 175, pp. 177--211, 2010.

\bibitem{FacchineiPang2007}
F.~Facchinei and J.~Pang, \emph{Finite-dimensional variational inequalities and
  complementarity problems}.\hskip 1em plus 0.5em minus 0.4em\relax Springer
  Science \& Business Media, 2007.

\bibitem{Khalil}
H.~K. Khalil, \emph{Nonlinear Systems}, ser. Pearson Education.\hskip 1em plus
  0.5em minus 0.4em\relax Prentice Hall, 2002.

\end{thebibliography}
\end{document}